\documentclass{article}

\usepackage{layout}
\usepackage{graphicx}
\usepackage{lmodern}
\usepackage{amsmath,amsthm,amssymb,mathrsfs,mathtools,bm,eucal,tensor,amscd}
\usepackage[utf8]{inputenc}
\usepackage[T1]{fontenc}
\usepackage{microtype}
\usepackage{calligra}
\usepackage[scaled]{beramono,berasans}
\usepackage[all,cmtip]{xy}
\usepackage[utf8]{inputenc}
\usepackage[T1]{fontenc}
\usepackage[margin=2.5cm,a4paper]{geometry}
\usepackage{titlesec}
\usepackage{appendix}
\usepackage{multicol}
\usepackage{multirow}
\usepackage{parskip}
\usepackage{tabto}
\usepackage{enumerate,comment,braket,xspace,tikz-cd}
\usepackage[hidelinks]{hyperref}
\usepackage{cleveref}
\usepackage{faktor}
\usepackage{vmargin}
\usepackage{adjustbox}
\usepackage{bbm}
\usepackage[
style=alphabetic
]{biblatex}
\usepackage{vmargin}
\setpapersize{A4}
\setmarginsrb{30mm}{10mm}{30mm}{10mm}%
{12mm}{10mm}{5mm}{10mm}

\theoremstyle{plain}
\newtheorem{thm-intro}{Theorem}
\newtheorem{thm}{Theorem}
\newtheorem{lem}[thm]{Lemma}
\newtheorem{prop}[thm]{Proposition}

\newtheorem{cor}[thm]{Corollary}

\theoremstyle{definition}

\newtheorem{eg}[thm]{Example}
\theoremstyle{remark}
\newtheorem*{rem*}{Remark}
\newtheorem{rem}[thm]{Remark}
\numberwithin{thm}{section}

\newcommand{\Aut}{\mathrm{Aut}}
\newcommand{\Hom}{\mathrm{Hom}}
\newcommand{\End}{\mathrm{End}}
\newcommand{\Gal}{\mathrm{Gal}}
\newcommand{\tsigma}{\tilde{\sigma}}
\newcommand{\ttau}{\tilde{\tau}}
\newcommand{\tlx}{T_{\ell}X}
\newcommand{\vlx}{V_{\ell}X}
\newcommand{\tlsigmax}{T_{\ell}{}^{\sigma}X}
\newcommand{\vlsigmax}{V_{\ell}{}^{\sigma}X}
\newcommand{\sigmax}{{}^{\sigma}X}

\title{Descent properties for an abelian variety\\ with extended Galois representation}
\author{Ludovic Felder}
\date{}

\begin{document}

\maketitle

\begin{abstract}
    Let $K$ be a field, $L$ a finite Galois extension of $K$, and $X$ an abelian variety defined over $L$. If $X$ is isogenous over $L$ to an abelian variety defined over $K$, then the $\ell$-adic Galois representations associated to $X$ extend to representations $\bar{\rho}_{\ell,X}:\Gal(\bar{L}/K)\to\Aut(V_\ell X)$ for every prime $\ell$. This paper aims to show that the converse is true for abelian varieties of Type I, with some supplementary conditions needed on the endomorphisms of $X$, when $L$ is either a number field or a function field of prime characteristic different from $2$.
\end{abstract}

\section*{Introduction}

Let $K$ be a field, let $L$ be a finite Galois field extension of $K$. We fix a separable closure $\bar{K}=\bar{L}$ for the rest of the article (the choice of such a separable closure does not matter). Denote by $G_L:=\Gal(\bar{L}/L)$ and $G_K:=\Gal(\bar{L}/K)$ the absolute Galois groups of $L$ and $K$. Let $X$ be an abelian variety defined over $L$, consider the $\ell$-adic Tate module $V_\ell X$ for any prime $\ell\neq\mathrm{char}(L)$, that is endowed with an $\ell$-adic representation $\rho_{\ell,X}:G_L\to\Aut(V_\ell A)$. If there exists an $L/K$-descent datum for $X$, then $X$ is isomorphic, or at least isogenous over $L$ to an abelian variety defined over $K$, and the natural action of the group $G_L$ on $V_\ell X$ extends to an action of the group $G_K$ for every prime $\ell\neq\mathrm{char}(L)$. This paper aims to study the converse question: if there is an action of the Galois group $G_K$ on $V_\ell X$ that restricts to the usual action of $G_L$ on $V_\ell X$ for one prime $\ell$ or several, can we find an $L/K$-descent datum for $X$? In the first section we will mainly recall some definitions about Galois representations associated to an abelian variety and the Galois conjugates $\sigmax$ of $X$ for $\sigma$ in $\Gal(L/K)$, as well as the most important notations. In the second section we will assume that $L$ is either a finite field, a function field of positive characteristic different from $2$ or a number field. We prove the following result:
\begin{thm}
If the representation $\rho_{\ell,X}:G_L\to\Aut(V_\ell X)$ can be extended to a representation $\bar{\rho}_{\ell,X}:G_K\to\Aut(V_\ell X)$ for a prime $\ell$ different from $\mathrm{char}(L)$, then $X$ is a \emph{$K$-abelian variety}, that means an abelian variety defined over $L$ that is isogenous to each of its $\Gal(L/K)$-conjugates.
\end{thm}
From this point onwards, we only consider abelian varieties of type I, i.e. abelian varieties whose endomorphism algebra is a totally real field $F$; in particular this rules out the case of an abelian variety defined over a finite field. In the third section we show that the result of \cite{ribet_1994} concerning the field of definition of a $K$-abelian variety holds in greater generality than originally stated and derive some consequences for our case of study. Concretely, the result can be stated as:
\begin{thm}
    Assume that the representation $\rho_{\ell,X}:G_L\to \Aut(V_\ell X)$ can be extended to a representation $\bar{\rho}_{\ell,X}:G_K\to\Aut(V_\ell X)$ that commutes with the endomorphisms of $X$, for a prime $\ell$ different from $\mathrm{char}(L)$. Moreover, assume that $\End_L(X)=\End_{\bar{L}}(X)=F$ is a totally real field, i.e. $X$ is absolutely simple of type I, and that the order of the Galois group $\Gal(L/K)$ is odd. Then $X$ is isogenous over $L$ to an abelian variety defined over $K$.
\end{thm}
Finally, the goal of the last section is to study the general case, where we do not have any restriction on the order of the group $\Gal(L/K)$. The main result of the paper is:
\begin{thm}
    Assume that the representation $\rho_{\ell,X}:G_L\to \Aut(V_\ell X)$ can be extended to a representation $\bar{\rho}_{\ell,X}:G_K\to\Aut(V_\ell X)$ that commutes with the endomorphisms of $X$, for every prime $\ell$ except for a finite set of primes that does not contain $2$. Then, under some mild conditions on $F$, $X$ is isogenous over $L$ to an abelian variety defined over $K$.
\end{thm}

\section{Definitions}\label{definition_section}

\subsection{The Galois representations associated to an abelian variety} \label{Gal_rep_section}

Let $L$ be a field, $X$ an abelian variety of dimension $g$ defined over $L$. Fix a separable closure $\bar{L}$ of $L$, denote by $G_L:=\Gal(\bar{L}/L)$ the absolute Galois group of $L$. If $Y$ is another abelian variety defined over $L$, denote by $\Hom_L(Y,X)$ the morphisms of abelian varieties from $Y$ to $X$ defined over $L$, this is a free $\mathbb{Z}$-module of finite type. We define the \emph{morphisms up to isogeny} as $\Hom^0_L(Y,X):=\Hom_L(Y,X)\otimes_{\mathbb{Z}}\mathbb{Q}$, and we will say by abuse of notation that an invertible element in $\Hom_L^0(Y,X)$ is an \emph{isogeny} from $Y$ to $X$. We denote by $\End_L(X)$ the endomorphisms of $X$ defined over $L$, and by $\End^0_L(X)$ the endomorphisms up to isogeny.
\newline For every rational prime $\ell$ different from $\mathrm{char}(L)$, for every integer $n\geq 0$, denote by $X[\ell^n]$ the kernel of the multiplication map $[\ell^n]_X:X\to X$. It is known that the group $X[\ell^n](\bar{L)}$ of $\bar{L}$-points of $\ell^n$-torsion is isomorphic to $(\mathbb{Z}/\ell^n\mathbb{Z})^{2g}$, and that the groups $\{X[\ell^n](\bar{L})\}_{n\geq 0}$ form a projective system where the transition maps are given by the multiplication $[\ell]_X:X[\ell^{n+1}](\bar{L})\to X[\ell^n](\bar{L})$. By taking the projective limit, we obtain the $\ell$-adic Tate module
\begin{align} T_{\ell}X=\underleftarrow{\mathrm{lim}}(X[\ell^n](\bar{L})).\end{align}It is non-canonically isomorphic to $\mathbb{Z}_\ell^{2g}$. The group $G_L$ acts on the left on $X[\ell^n](\bar{L})$ for every $n\geq 0$, by sending a point $Q\in X[\ell^n](\bar{L})$ to the point $Q\circ\mathrm{Spec}(\Tilde{\sigma})\in X[\ell^n](\bar{L})$, for $\Tilde{\sigma}\in G_L
$. For every natural integer $n$, this defines a map $\rho_{X,n}:G_L\to\Aut(X[\ell^n](\bar{L}))$ that makes this diagram commutative:
\begin{equation}\begin{tikzcd}
{X[\ell^n]} \arrow[r, no head, equal]                                      & {X[\ell^n]}                                                       \\
\mathrm{Spec}(\bar{L}) \arrow[r, "\mathrm{Spec}(\tsigma)^{-1}"'] \arrow[u, "Q"] & \mathrm{Spec}(\bar{L}) \arrow[u, "\rho_{X,n}(\tsigma)(Q)"']
\end{tikzcd}\end{equation}
The maps $\rho_{X,n}$ commute with the multiplication map $[\ell]_X$, in the sense that for every positive integer $n$, we have $[\ell]_X\circ\rho_{X,n+1}=\rho_{X,n}\circ [\ell]_X$. As a consequence the Tate module is endowed with a continuous Galois representation \begin{align}\rho_{\ell,X}:G_L\to \Aut(T_{\ell}X)\end{align}obtained by considering compatible systems of points of $\{X[\ell^n](\bar{L})\}_{n\in\mathbb{N}}$  (by this we mean a set of points $\{Q_n\in X[\ell^n](\bar{L})\}_{n\in\mathbb{N}}$ such that $[l]_X (Q_{n+1})=Q_n$ for every $n$) and putting together the $\rho_{X,n}$'s, for all $n\geq 0$. Write $V_\ell X:=T_\ell X\otimes_{\mathbb{Z}_\ell}\mathbb{Q}_\ell$, this is a $\mathbb{Q}_\ell$-vector space of dimension $2g$; the preceding constructions induce a Galois representation that we are going to denote the same way, $\rho_{\ell,X}:G_L\to\Aut(V_\ell X)$. For $\tsigma\in G_L$, we will see $\rho_{\ell,X}(\tsigma)$ as an automorphism of $\tlx$ or $\vlx$ depending on the context.
\newline The construction of the Tate module is functorial: if $g:Y\to X$ is a morphism of abelian varieties over $L$, then $g$ induces a $G_L$-equivariant, $\mathbb{Z}_\ell$-linear homomorphism of Tate modules $T_\ell g:T_\ell Y\to T_\ell X$ by sending the element $(0,Q_1,Q_2,...)$ of $T_\ell Y$ to the element $(0,g(Q_1),g(Q_2),...)$ of $T_\ell X$. Similarly, a morphism up to isogeny $g\in\Hom_L^0(Y,X)$ induces a $G_L$-equivariant, $\mathbb{Q}_\ell$-linear homomorphism of Tate modules $V_\ell g:V_\ell Y\to V_\ell X$. More precisely, the natural map
\begin{equation}\begin{split}
        \Phi :\Hom_L(Y,X)\otimes_{\mathbb{Z}}\mathbb{Z}_{\ell}&\rightarrow \Hom_{G_L}(T_{\ell}Y,\tlx)\\
        f\otimes \alpha&\mapsto \alpha\cdot T_{\ell}f,\end{split}
    \end{equation}is injective for every abelian varieties $X$ and $Y$ over $L$, where $\Hom_{G_L}(T_\ell Y,\tlx)$ denotes the homomorphisms of Tate modules commuting with $G_L$. In particular, for $Y=X$, the action of $G_L$ on $\tlx$ (resp. on $\vlx$) commutes with $\End_L(X)$ (resp. $\End_L^0(X)$).

\subsection{Galois conjugates of abelian varieties} \label{twist_section}

Let $K$ be a subfield of $L$, such that $L$ is a finite Galois extension of $K$. Fix a separable closure $\bar{K}=\bar{L}$, so that $G_L$ can be identified with a normal subgroup of $G_K:=\Gal(\bar{K}/K)$. Let $\sigma\in\Gal(L/K)$. We define the Galois conjugate ${}^{\sigma}X$ of $X$ as the fibered product

\begin{equation}\label{commdiag} \begin{CD}
    {}^{\sigma}X @>>> X\\
    @VVV @VV f V \\
    \mathrm{Spec}(L) @>>\mathrm{Spec}(\sigma)> \mathrm{Spec}(L),
    \end{CD} \end{equation}

\medskip
\noindent with $f$ being the structure morphism of $X$. Concretely, if $X$ is defined by a set of polynomials $\{g_1,...,g_n\}\subset L[X_1,...,X_m]$, then ${}^{\sigma}X$ is the abelian variety defined by the set of polynomials $\{\sigma(g_1),...,\sigma(g_n)\}\subset L[X_1,...,X_m]$. The top arrow is an isomorphism of schemes, that we will denote by $\sigma_X$. In fact the varieties $X$ and ${}^{\sigma}X$ can be seen as the same scheme, with different structure morphisms: namely, in this diagram the structure morphism for $\sigmax$ is simply $\mathrm{Spec}(\sigma)^{-1}\circ f$ when identifying $\sigmax$ and $X$ via $\sigma_X$. This can be seen in the following diagram, by showing that the outer square enjoys the universal property of the fibered product:

\begin{equation}\label{identification} \begin{tikzcd}
X \arrow[d, "f"'] \arrow[r, no head, equal]                & X \arrow[dd, "f"]      \\
\mathrm{Spec}(L) \arrow[d, "\mathrm{Spec}(\sigma)^{-1}"'] &                        \\
\mathrm{Spec}(L) \arrow[r, "\mathrm{Spec}(\sigma)"'] & \mathrm{Spec}(L).
\end{tikzcd} \end{equation}

\medskip
\noindent If $\tau\in\Gal(L/K)$, the equality $\mathrm{Spec}(\sigma\tau)=\mathrm{Spec}(\tau)\circ\mathrm{Spec}(\sigma)$ shows that ${}^{\sigma}({}^{\tau}X)={}^{\sigma\tau}X$ by noticing that the outer squares of the following diagrams are the same, because the inner squares are cartesian:

\medskip
\begin{center} $\begin{CD}
    {}^{\sigma}({}^{\tau}X) @>>> {}^{\tau}X @>>> X\\
    @VVV @VVV @VV f V \\
    \mathrm{Spec}(L) @>>\mathrm{Spec}(\sigma)> \mathrm{Spec}(L) @>>\mathrm{Spec}(\tau)> \mathrm{Spec}(L)
    \end{CD}$ 
    \hspace{5mm} and \hspace{5mm}
    $\begin{CD}
    {}^{\sigma\tau}X @>>> X\\
    @VVV @VV f V \\
    \mathrm{Spec}(L) @>>\mathrm{Spec}(\sigma\tau)> \mathrm{Spec}(L).
    \end{CD}$
    \end{center}

\medskip
\noindent The construction of $\sigmax$ is functorial, in the sense that, if $g:Y\to X$ is a morphism of abelian varieties over $L$, there exists a "conjugated" morphism ${}^{\sigma}g:{}^{\sigma}Y\to{}^{\sigma}X$ over $L$ given by the universal property, and explicitly defined as \begin{align}{}^{\sigma}g=\sigma^{-1}_X\circ g\circ \sigma_Y.\end{align}We check easily that ${}^{\sigma}g$ is still a morphism of abelian varieties over $L$ and that ${}^{\sigma}({}^{\tau}g)={}^{\sigma\tau}g$. The maps ${}^{\sigma}g$ and $g$ are the same when we identify $\sigmax$ and $X$ via $\sigma_X$ (resp. ${}^{\sigma}Y$ and $Y$ via $\sigma_Y$), so the free $\mathbb{Z}$-modules $\Hom(Y,X)$ and $\Hom({}^{\sigma}Y,\sigmax)$ are identified by the map sending $g$ to ${}^{\sigma}g$. In particular the rings $\End(X)$ and $\End(\sigmax)$ are the same and we can identify an element $f$ of $\End(X)$ with its Galois conjugate ${}^{\sigma}f$ in $\End(\sigmax)$. As ${}^{\sigma}g$ is a morphism of abelian varieties over $L$, this construction induces a homomorphism of Tate modules $T_{\ell}{}^{\sigma}g: T_{\ell}{}^{\sigma}Y\to\tlsigmax$. If $g\in\Hom^0(Y,X)$, we define in the same way the conjugated morphism ${}^{\sigma}g\in\Hom^0({}^{\sigma}Y,\sigmax)$ and the associated morphism of Tate modules $V_\ell {}^{\sigma}g:V_\ell {}^{\sigma}Y\to V_\ell {}^{\sigma}X$. Notice that, although we identified the endomorphisms of $X$ with the endomorphisms of $\sigmax$, we do not know in general if there exists any morphism of abelian varieties from $\sigmax$ to $X$, let alone if $\sigmax$ and $X$ are isogenous.

\begin{rem}
    If there exists an abelian variety $X_K$ defined over $K$ such that $X=X_K\times_K L$, then there is an isomorphism of abelian varieties over $L$ from every Galois conjugate $\sigmax$ to $X$. This comes from the fact that the polynomial equations defining $X$ have coefficients in $K$.
\end{rem}

\subsection{The action on the torsion points} \label{pi_section}

We can extend the map $\rho_{X,n}:G_L\to\Aut(X[\ell^n](\bar{L}))$ defined in section \ref{Gal_rep_section} to a map
\begin{equation}\pi_{X,n}:G_K\to \bigsqcup_{\sigma\in\Gal(L/K)}\Hom(X[\ell^n](\bar{L}),\sigmax[\ell^n](\bar{L})),\end{equation} where $\pi_{X,n}(\tsigma)\in\Hom(X[\ell^n](\bar{L}),{}^{\tsigma_{\mid L}}X[\ell^n](\bar{L}))$, for $\tsigma$ in $G_K$. Explicitly, for $\Tilde{\sigma}\in G_K$ and $\sigma=\tsigma_{\mid L}$ its restriction to $L$, for $Q\in X[\ell^n](\bar{L})$,
\begin{align}\pi_{X,n}(\Tilde{\sigma})(Q)=P:=\sigma_X^{-1}\circ Q\circ\mathrm{Spec}(\Tilde{\sigma})\in {}^{\sigma}X[\ell^n](\bar{L}).\end{align}Thus the following diagram commutes:
\medskip
\begin{equation} \begin{tikzcd}
    & \sigmax[\ell^n] \arrow[rr, "\sigma_X"] \arrow[dd]  &   & X[\ell^n] \arrow[dd]    \\
\mathrm{Spec}(\bar{L}) \arrow[rr, "\mathrm{Spec}(\tilde{\sigma})"] \arrow[ru, "P=\pi_{X,n}(\tsigma)(Q)"] \arrow[rd] &   & \mathrm{Spec}(\bar{L}) \arrow[ru, "Q"] \arrow[rd] &    \\
        & \mathrm{Spec}(L) \arrow[rr, "\mathrm{Spec}(\sigma)"] &                                                   & \mathrm{Spec}(L)
\end{tikzcd} \end{equation}

\medskip
\noindent We check easily that $P$ is $L$-linear because $\sigma_X$ and $\mathrm{Spec}(\tsigma)$ both are $\sigma$-semi-linear. As the maps $\sigma_X:\sigmax\to X$ and $\mathrm{Spec}(\tsigma):\mathrm{Spec}(\bar{L})\to\mathrm{Spec}(\bar{L})$ are isomorphisms of schemes, the map $\pi_{X,n}(\tsigma)$ is an isomorphism of groups for every $\tsigma\in G_K$. The $\pi_{X,n}(\tsigma)$'s once again commute with the multiplication-by-$\ell$ map and we define by taking the projective limit of any compatible system of points of $\{X[\ell^n](\bar{L})\}_{n\geq 0}$, an isomorphism of $\mathbb{Z}_\ell$-modules \begin{equation}\pi_X(\tsigma):\tlx\to\tlsigmax.\end{equation}If we take $\Tilde{\sigma}\in G_L$, i.e. $\Tilde{\sigma}_{\mid L}=\sigma=\mathrm{id}_L$, then $\sigma_X=\mathrm{Id}_X$ and $\pi_X(\Tilde{\sigma})=\rho_{\ell,X}(\Tilde{\sigma})$, hence justifying that $\pi_X$ is an extension of $\rho_{\ell,X}$.

\begin{rem}
From now on, we will denote by $\Tilde{\sigma}$, $\Tilde{\tau}$... the elements of the absolute Galois group $G_K$ and by $\sigma$, $\tau$... their restrictions to $L$, i.e. elements of the finite Galois group $\Gal(L/K)$. As $G_L$ is a subgroup of $G_K$, we will write $\Tilde{\sigma}$, $\Tilde{\tau}$... for the elements of $G_L$ too.
\end{rem}

If $\Tilde{\sigma}$, $\Tilde{\tau}$ are in $G_K$, then for every positive integer $n$, for every $Q$ in $X[\ell^n](\bar{L})$ we have
\begin{align*} \pi_{X,n}(\Tilde{\sigma}\Tilde{\tau})(Q)&=(\sigma\tau)_X^{-1}\circ Q\circ\mathrm{Spec}(\Tilde{\sigma}\Tilde{\tau})\\
&=\sigma_{{}^{\tau}X}^{-1}\circ\tau_X^{-1}\circ Q\circ\mathrm{Spec}(\Tilde{\tau})\circ\mathrm{Spec}(\Tilde{\sigma})\\
&=\sigma_{{}^{\tau}X}^{-1}\circ \pi_{X,n}(\Tilde{\tau})(Q)\circ\mathrm{Spec}(\Tilde{\sigma})\\
&=\pi_{{}^\tau X,n}(\Tilde{\sigma})(\pi_{X,n}(\Tilde{\tau})(Q)) \end{align*}
and this gives \begin{align}\label{compo}\pi_X(\Tilde{\sigma}\Tilde{\tau})=\pi_{{}^{\tau}X}(\Tilde{\sigma})\circ\pi_X(\Tilde{\tau})\end{align} on the Tate modules (notice that $\pi_X(\tsigma)^{-1}=\pi_{{}^{\sigma}X}(\tsigma^{-1})$ because $\pi_X(\mathrm{id})=\pi_X(\tsigma^{-1}\tsigma)=\mathrm{Id}_{T_\ell X}$). We naturally extend this map to a $\mathbb{Q}_\ell$-linear isomorphism $\pi_X(\tsigma):V_\ell X\to V_\ell\sigmax$, that enjoys the property (\ref{compo}) and that coincides with $\rho_{\ell,X}(\tsigma)\in\Aut(\vlx)$ when $\tsigma$ lies in $G_L$. In what follows, the context will determine if we are seeing $\pi_X(\tsigma)$ as an element of $\Hom(\tlx,\tlsigmax)$ or $\Hom(\vlx,\vlsigmax)$.
\newline If $\phi: T_\ell Y\to\tlx$ (resp. $\phi:V_\ell Y\to\vlx$) is a $G_L$-equivariant homomorphism between the Tate modules of two abelian varieties $Y$ and $X$, define a "conjugated" morphism ${}^{\sigma}\phi:T_\ell {}^{\sigma}Y\to\tlsigmax$ (resp. ${}^{\sigma}\phi:V_\ell {}^{\sigma}Y\to\vlsigmax$) by the formula
\begin{equation}\label{formula}
    {}^{\sigma}\phi=\pi_X(\tsigma)\circ\phi\circ\pi_Y(\tsigma)^{-1}
\end{equation} for any $\tsigma\in G_K$ extending $\sigma$. This does not depend on the extension of $\sigma$ in $G_K$; indeed, if $\tsigma_1$, $\tsigma_2$ are two elements of $G_K$ extending $\sigma$ (so that the composition $\tsigma_2^{-1}\tsigma_1$ lies in the group $G_L$), then
\begin{align*}
    {}^{\sigma}\phi&=\pi_X(\tsigma_1)\circ \phi\circ\pi_Y(\tsigma_1)^{-1}\\ &=\pi_X(\tsigma_2\tsigma_2^{-1}\tsigma_1)\circ \phi\circ\pi_Y(\tsigma_2\tsigma_2^{-1}\tsigma_1)^{-1}\\ &=\pi_X(\tsigma_2)\circ\pi_X(\tsigma_2^{-1}\tsigma_1)\circ \phi\circ\pi_Y(\tsigma_2^{-1}\tsigma_1)^{-1}\circ\pi_Y(\tsigma_2)^{-1}\\ &=\pi_X(\tsigma_2)\circ\rho_{\ell,X}(\tsigma_2^{-1}\tsigma_1)\circ \phi\circ \rho_{\ell,Y}(\tsigma_2^{-1}\tsigma_1)^{-1}\circ\pi_Y(\tsigma_2)^{-1}\\ &=\pi_X(\tsigma_2)\circ \phi\circ\rho_{\ell,Y}(\tsigma_2^{-1}\tsigma_1)\circ \rho_{\ell,Y}(\tsigma_2^{-1}\tsigma_1)^{-1}\circ\pi_Y(\tsigma_2)^{-1}\\ &=\pi_X(\tsigma_2)\circ\phi\circ\pi_Y(\tsigma_2)^{-1}
\end{align*}
because the morphism $\phi$ commutes with the action of $G_L$. It follows that ${}^{\sigma}\phi$ only depends on $\sigma:=\tsigma_{\mid L}$ as expected. The next proposition justifies this formula:

\begin{prop} \label{twist_tate_morphism} If there exists a morphism of abelian varieties $g:Y\to X$ over $L$ such that $\phi=T_\ell g$ (resp. an element $g\in\Hom^0(Y,X)$ such that $\phi=V_\ell g$), then \begin{align*} {}^{\sigma}\phi&=T_{\ell}{}^{\sigma}g\\
\text{(resp. }{}^{\sigma}\phi&=V_\ell {}^{\sigma}g \text{).}\end{align*} Thus ${}^{\sigma}T_\ell g=T_\ell {}^{\sigma}g=\pi_X(\tsigma)\circ T_\ell g\circ \pi_Y(\tsigma)^{-1}$ (resp. ${}^{\sigma}V_\ell g=V_\ell {}^{\sigma}g=\pi_X(\tsigma)\circ V_\ell g\circ \pi_Y(\tsigma)^{-1}$) for any $\tsigma$ in $G_K$ extending $\sigma$. \end{prop}

\begin{proof} In the following commutative diagram, let $P\in {}^{\sigma}Y[\ell^n](\bar{L})$ and $Q=\pi_{Y,n}(\tsigma)^{-1}(P)$. The composition of the three arrows on top of the diagram is ${}^{\sigma}g$ (restricted to the points of $\ell^n$-torsion):
\begin{center} $\begin{CD}
    {}^{\sigma}Y[\ell^n] @> \sigma_Y >> Y[\ell^n] @> g >> X[\ell^n] @> \sigma_X^{-1} >> {}^{\sigma}X[\ell^n] \\
    @A P AA @A Q AA @AA g\circ Q A @AA {}^{\sigma}g\circ P A \\
    \mathrm{Spec}(\bar{L}) @>>\mathrm{Spec}(\bar{\sigma})> \mathrm{Spec}(\bar{L}) @= \mathrm{Spec}(\bar{L}) @>>\mathrm{Spec}(\bar{\sigma})^{-1}> \mathrm{Spec}(\bar{L})
    \end{CD}$ \end{center}
so we get by diagram-chasing \begin{align*} {}^{\sigma}g\circ P &= \sigma_X^{-1}\circ g \circ Q \circ\mathrm{Spec}(\tsigma)\\ &=\sigma_X^{-1}\circ g\circ\pi_{Y,n}(\tsigma)^{-1}(P)\circ\mathrm{Spec}(\tsigma)\\ &=\pi_{X,n}(\tsigma)(g\circ\pi_{Y,n}(\tsigma)^{-1}(P))\end{align*}
and this gives us the result we were looking for by taking projective limits.\end{proof}

\begin{prop}
    \label{twistgal} The natural representation ${}^{\sigma}\rho_{\ell,X}:G_L\to\Aut(\tlsigmax)$ associated to $\sigmax$ (as well as ${}^{\sigma}\rho_{\ell,X}:G_L\to\Aut(V_\ell {}^{\sigma}X)$) satisfies the equality \begin{align}\label{rho_tordu} {}^{\sigma}\rho_{\ell,X}(\ttau)=\pi_X(\tsigma)\circ\rho_{\ell,X}(\tsigma^{-1}\ttau\tsigma)\circ\pi_X(\tsigma)^{-1}\end{align} for any $\tsigma$ in $G_K$ that extends $\sigma$, for any $\ttau$ in $G_L$.
\end{prop}

\begin{proof}
    First, let us notice that, unless the group $G_L$ is abelian, in general $\rho_{\ell,X}(\tsigma\ttau)\neq\rho_{\ell,X}(\ttau\tsigma)$, and the formula (\ref{formula}) does not hold. We prove the formula (\ref{rho_tordu}) by looking once again at the points of $\ell^n$-torsion: for $P\in\sigmax[\ell^n](\bar{L})$, $\ttau\in G_L$, the natural action of $\ttau$ on $P$ is given by
\begin{align*}
    {}^{\sigma}\rho_{X,n}(\ttau)(P)=P\circ\mathrm{Spec}(\ttau).
\end{align*}Choose $Q_1$ and $Q_2$ in $X[\ell^n](\bar{L})$ and $\tsigma$ in $G_K$ that restricts to $\sigma$ on $L$ such that the following diagram commutes, where the middle square corresponds to the action of $\ttau$ on $P$: 
\begin{center} $\begin{CD}
    X[\ell^n] @> \sigma_X^{-1} >> \sigmax[\ell^n] @= \sigmax[\ell^n] @> \sigma_X >> X[\ell^n] \\
    @A Q_1 AA @A P AA @AA P\circ\mathrm{Spec}(\ttau) A @AA Q_2 A \\
    \mathrm{Spec}(\bar{L}) @>>\mathrm{Spec}(\bar{\sigma})^{-1}> \mathrm{Spec}(\bar{L}) @>> \mathrm{Spec}(\ttau)^{-1} > \mathrm{Spec}(\bar{L}) @>>\mathrm{Spec}(\bar{\sigma})> \mathrm{Spec}(\bar{L})
    \end{CD}$ \end{center}
First of all, by following the arrows, we have $Q_1=\pi_{X,n}(\tsigma)^{-1}(P)=\sigma_X\circ P\circ \mathrm{Spec}(\tsigma)^{-1}$ and $Q_2=\pi_{X,n}(\tsigma)^{-1}(P\circ\mathrm{Spec}(\ttau))=\sigma_X\circ (P\circ\mathrm{Spec}(\ttau))\circ \mathrm{Spec}(\tsigma)^{-1}$. This amounts to \begin{align*}Q_2=Q_1\circ \mathrm{Spec}(\tsigma)\circ\mathrm{Spec}(\ttau)\circ\mathrm{Spec}(\tsigma)^{-1}=Q_1\circ\mathrm{Spec}(\tsigma^{-1}\ttau\tsigma),\end{align*} so $Q_2=\rho_{X,n}(\tsigma^{-1}\ttau\tsigma)(Q_1)$ in $X[\ell^n](\bar{L})$, and the composition $\tsigma^{-1}\ttau\tsigma$ lies in $G_L$ because $G_L$ is a normal subgroup of $G_K$. Then
\begin{align*}
    {}^{\sigma}\rho_{\ell,n}(\ttau)(P)&=\sigma_X^{-1}\circ Q_2\circ\mathrm{Spec}(\tsigma)\\
    &=\sigma_X^{-1}\circ Q_1\circ\mathrm{Spec}(\tsigma^{-1}\ttau\tsigma)\circ\mathrm{Spec}(\tsigma)\\
    &=\sigma_X^{-1}\circ\sigma_X\circ P\circ\mathrm{Spec}(\tsigma)^{-1}\circ\mathrm{Spec}(\tsigma^{-1}\ttau\tsigma)\circ\mathrm{Spec}(\tsigma)\\
    &=\sigma_X^{-1}\circ\pi_{X,n}(\tsigma)^{-1}(P)\circ\mathrm{Spec}(\tsigma^{-1}\ttau\tsigma)\circ\mathrm{Spec}(\tsigma)\\
    &=\sigma_X^{-1}\circ\rho_{X,n}(\tsigma^{-1}\ttau\tsigma)(\pi_{X,n}(\tsigma)^{-1}(P))\circ\mathrm{Spec}(\tsigma)\\
    &=\pi_{X,n}(\tsigma)(\rho_{X,n}(\tsigma^{-1}\ttau\tsigma)(\pi_{X,n}(\tsigma)^{-1}(P))).
\end{align*} Finally, on the Tate modules, we get:
\begin{align*}{}^{\sigma}\rho_{\ell,X}(\ttau)=\pi_X(\tsigma)\circ\rho_{\ell,X}(\tsigma^{-1}\ttau\tsigma)\circ\pi_X(\tsigma)^{-1}.
\end{align*}\end{proof}

\begin{rem}
    It is obvious that this expression does not depend on the extension of $\sigma$ to $G_K$ because the left hand side does not, but we can still show it by hand easily in the same manner as in formula (\ref{formula}).
\end{rem}

\section{Extension of the representation}\label{extension_section}

We say that $X$ \emph{descends} to $K$ if there exists an abelian variety $Y_K$ defined over $K$ such that $Y:=Y_K\times_K L$ is isomorphic to $X$ over $L$. The following theorem was proved in \cite{weil_1956}:
\begin{thm}[Weil]\label{weil}
    Assume that, for every $\sigma$ in $\Gal(L/K)$, there is an isomorphism of abelian varieties $h(\sigma):\sigmax\tilde{\to}X$ over $L$, satisfying the cocycle condition $h(\sigma\tau)=h(\sigma)\circ {}^{\sigma}h(\tau)$ for every $\sigma$, $\tau$ in $\Gal(L/K)$. Then $X$ descends to $K$.
\end{thm}
There exists another version of this theorem, (see \cite{ribet_2004}, Theorem (8.2)), by considering isogenies instead of isomorphisms. We say that $X$ \emph{descends up to isogeny} to $K$ if there exists an abelian variety $Y_K$ defined over $K$ such that $Y:=Y_K\times_K L$ is isogenous to $X$ over $L$.
\begin{thm}[Ribet]\label{descente_isogenie}
    Assume that, for every $\sigma$ in $\Gal(L/K)$, there is an isogeny $h(\sigma)$ in $\Hom_L^0(\sigmax,X)$, satisfying the cocycle condition $h(\sigma\tau)=h(\sigma)\circ {}^{\sigma}h(\tau)$ for every $\sigma$, $\tau$ in $\Gal(L/K)$. Then $X$ descends up to isogeny to $K$.
\end{thm}

Suppose that, for every $\sigma$ in $\Gal(L/K)$, we have an isomorphism $h(\sigma):\sigmax\tilde{\to} X$ over $L$, satisfying the conditions of Theorem \ref{weil}. Then, for every prime $\ell$ different from $\mathrm{char}(L)$, we can extend the action of $G_L$ on $\tlx$ (and thus on $\vlx$) to an action of $G_K$ by setting, for every $\tsigma$ in $ G_K$, \begin{equation}\bar{\rho}_{\ell,X}(\tsigma)=T_{\ell}h(\sigma)\circ\pi_X(\tsigma).\end{equation}We check easily by using Proposition \ref{twist_tate_morphism} that this defines a homomorphism of groups:
\begin{align*}
    \bar{\rho}_{\ell,X}(\tsigma\ttau)&=T_{\ell}h(\sigma\tau)\circ\pi_X(\tsigma\ttau)\\
    &=T_{\ell}h(\sigma)\circ T_{\ell}{}^{\sigma}h(\tau)\circ \pi_{{}^{\tau}X}(\tsigma)\circ\pi_X(\ttau)\\
    &=T_{\ell}h(\sigma)\circ\pi_X(\tsigma)\circ T_{\ell}h(\tau)\circ\pi_X(\ttau)\\
    &=\bar{\rho}_{\ell,X}(\tsigma)\circ\bar{\rho}_{\ell,X}(\ttau),
\end{align*}
and that, for $\tsigma$ in $G_L$, i.e. satisfying $\tsigma_{\mid L}=\mathrm{id}_L$, the automorphisms $\bar{\rho}_{\ell,X}(\tsigma)$ and $\rho_{\ell,X}(\tsigma)$ coincide. If the maps $h(\sigma)$ are only isogenies, we can only extend the Galois representation on $\Aut(V_\ell X)$ with the formula \begin{equation}\bar{\rho}_{\ell,X}(\tsigma)=V_\ell h(\sigma)\circ\pi_X(\tsigma).\end{equation}
\begin{rem}\label{important_remark}The extended Galois representation $\bar{\rho}_{\ell,X}$ is isomorphic to the the natural representation $\rho_{\ell,Y_K}:G_K\to\Aut(T_\ell Y_K)$ (resp. $\rho_{\ell,Y_K}:G_K\to\Aut(V_\ell Y_K)$): indeed, as a consequence of Theorem \ref{weil} (resp. Theorem \ref{descente_isogenie}), there exists an isomorphism (resp. an isogeny) of abelian varieties $f:X\to Y$ over $L$ such that $h(\sigma)=f^{-1}\circ {}^{\sigma}f$. Notice that $Y[\ell^n](\bar{L})=Y_K[\ell^n](\bar{L})$ because $Y$ is obtained from $Y_K$ by extension of scalars and the action of the group $G_K$ on $Y[\ell^n](\bar{L})$ is the same as on $Y_K[\ell^n](\bar{L})$. We can identify ${}^{\sigma}Y$ and $Y$ as abelian varieties over $L$ because $Y$ is the extension of scalars of an abelian variety defined over $K$. This gives the following commutative diagram on the $\bar{L}$-points of $\ell^n$-torsion:
\begin{equation}
    \begin{tikzcd} & {Y[\ell^n]} \arrow[r]  & {{}^{\sigma}Y[\ell^n]} \arrow[r, equal] & {Y[\ell^n]}  &  \\
\mathrm{Spec}(\bar{L}) \arrow[ru, "f\circ P"'] \arrow[rd, "P"] \arrow[rrrr, "\mathrm{Spec}(\tsigma)^{-1}", bend left=60, shift left=4] & &  &  & \mathrm{Spec}(\bar{L}) \arrow[lu, "f\circ P\circ \mathrm{Spec}(\tsigma)"'] \arrow[ld, "P\circ \mathrm{Spec}(\tsigma)"] \\
 & {X[\ell^n]} \arrow[r, "\sigma_X^{-1}"'] \arrow[uu, "f"] & {\sigmax[\ell^n]} \arrow[uu, "{}^{\sigma}f"]  & {X[\ell^n]} \arrow[uu, "f"'] &
\end{tikzcd}
\end{equation}
The homomorphism $h(\sigma)=f^{-1}\circ {}^{\sigma}f:\sigmax\to X$ completes this commutative diagram and it proves that the map that sends $\tsigma\in G_K$ to $T_\ell f\circ \bar{\rho}_{\ell,X}(\tsigma)\circ T_\ell f^{-1}$ (resp. $V_\ell f\circ \bar{\rho}_{\ell,X}(\tsigma)\circ V_\ell f^{-1}$) is equal to the natural representation $\rho_{\ell,Y_K}:G_K\to\Aut(T_\ell Y_K)$ (resp. $\rho_{\ell,Y_K}:G_K\to\Aut(V_\ell Y_K)$) by taking projective limits.\end{rem}

Let us now forget the isogenies of abelian varieties $h(\sigma)$ and suppose conversely that we can extend the $\ell$-adic Galois representation $\rho_{\ell,X}$ to the group $G_K$, for a rational prime $\ell$ different from $\mathrm{char}(L)$. We will consider two cases: we say that $X$ satisfies the condition $(*_\ell)$ if there exists a homomorphism of groups
\begin{align*}\tag{$*_\ell$}
    \bar{\rho}_{\ell,X}:G_K\to \Aut(V_{\ell}X)
\end{align*}
that restricts to $\rho_{\ell,X}$ on the subgroup $G_L$ of $G_K$, and commutes with $\End^0(X)$. We say that $X$ satisfies the condition $(**_\ell)$ if there exists a homomorphism of groups
\begin{align*}\tag{$**_\ell$}
    \bar{\rho}_{\ell,X}:G_K\to \Aut(T_{\ell}X)
\end{align*}
that restricts to $\rho_{\ell,X}$ on the subgroup $G_L$ of $G_K$, and commutes with $\End(X)$. Notice that the condition $(**_\ell)$ implies the condition $(*_\ell)$.
\newline Assume that $X$ satisfies the condition $(*_\ell)$. As $\pi_X(\tsigma)$ is an isomorphism of $\mathbb{Q}_\ell$-vector spaces from $\vlx$ to $\vlsigmax$, we can define an isomorphism of $\mathbb{Q}_\ell$-vector spaces $\eta(\tsigma):\vlsigmax\to\vlx$ as the composition $\eta(\tsigma)=\bar{\rho}_{\ell,X}(\tsigma)\circ\pi_X(\tsigma)^{-1}$ so that the following diagram is commutative for every $\tsigma$ in $G_K$:
\begin{equation}
    \begin{tikzcd}
V_\ell X \arrow[r, "\pi_X(\tsigma)"] \arrow[rr, "\bar{\rho}_{\ell,X}(\tsigma)"', bend right] & \vlsigmax \arrow[r, "\eta(\tsigma)"] & V_\ell X.
\end{tikzcd}
\end{equation}
If in particular $X$ satisfies the condition $(**_\ell)$, then the map $\eta(\tsigma)$ is an isomorphism of $\mathbb{Z}_\ell$-modules from $\tlsigmax$ to $\tlx$. We will view $\eta(\tsigma)$ as an element of $\Hom(\vlsigmax,\vlx)$ or $\Hom(\tlsigmax,\tlx)$ depending on whether $X$ satisfies condition $(*_\ell)$ or the less general condition $(**_\ell)$.

\begin{prop}\label{eta} The maps $\eta(\tsigma)$ only depend on the restriction $\sigma=\tsigma_{|L}$. As a consequence, from now on we will write $\eta(\sigma)$. \end{prop}
\begin{proof}[Proof] Let $\tsigma_1$, $\tsigma_2$ in $G_K$ such that both their restrictions to $L$ are equal to the same $\sigma\in\Gal(L/K)$. We have
\begin{align*}
    \eta(\tsigma_1) &=\eta(\tsigma_2\tsigma_2^{-1}\tsigma_1)\\
    &=\bar{\rho}_{\ell,X}(\tsigma_2\tsigma_2^{-1}\tsigma_1)\circ\pi_X(\tsigma_2\tsigma_2^{-1}\tsigma_1)^{-1}\\
    &=\bar{\rho}_{\ell,X}(\tsigma_2)\circ\bar{\rho}_{\ell,X}(\tsigma_2^{-1}\tsigma_1)\circ\pi_X(\tsigma_2^{-1}\tsigma_1)^{-1}\circ\pi_X(\tsigma_2)^{-1} \\
\end{align*}
As the product $\tsigma_2^{-1}\tsigma_1$ lies in $G_L$, we have $\bar{\rho}_{\ell,X}(\tsigma_2^{-1}\tsigma_1)=\pi_X(\tsigma_2^{-1}\tsigma_1)$. Thus we have the equality $\bar{\rho}_{\ell,X}(\tsigma_2^{-1}\tsigma_1)\circ\pi_X(\tsigma_2^{-1}\tsigma_1)^{-1}=\mathrm{Id}_{\vlx}$ and finally $\eta(\tsigma_1)=\bar{\rho}_{\ell,X}(\tsigma_2)\circ\pi_X(\tsigma_2)^{-1}=\eta(\tsigma_2)$. \end{proof}

\begin{prop} \label{commute}
The maps $\eta(\sigma)$ commute with the action of the group $G_L$ on the Tate modules, i.e. for every $\sigma$ in $\Gal(L/K)$ and every $\ttau$ in $G_L$, the following diagram commutes:
\begin{center} $\begin{CD}
    \vlsigmax @>\eta(\sigma)>> \vlx\\
    @V {}^{\sigma}\rho_{\ell,X}(\ttau) VV @VV \rho_{\ell,X}(\ttau) V \\
    \vlsigmax @>>\eta(\sigma)> \vlx
    \end{CD}$ \end{center}
\end{prop}

\begin{proof} 
This is a direct consequence of Proposition \ref{twistgal}: we have for any $\tsigma$ in $G_K$ extending $\sigma$
\begin{align*}
    {}^{\sigma}\rho_{\ell,X}(\ttau)&=\pi_X(\tsigma)\circ\rho_{\ell,X}(\tsigma^{-1}\ttau\tsigma)\circ\pi_X(\tsigma)^{-1}\\&=\pi_X(\tsigma)\circ\bar{\rho}_{\ell,X}(\tsigma)^{-1}\circ\rho_{\ell,X}(\ttau)\circ\bar{\rho}_{\ell,X}(\tsigma)\circ\pi_X(\tsigma)^{-1}\\&=\eta(\sigma)^{-1}\circ\rho_{\ell,X}(\ttau)\circ\eta(\sigma),  
\end{align*} and we have exactly what we wanted. \end{proof}

\begin{prop}\label{cocycle}
    The maps $\eta(\sigma):\vlsigmax\to\vlx$ satisfy the condition $\eta(\sigma\tau)=\eta(\sigma)\circ{}^{\sigma}\eta(\tau)$ for every $\sigma$, $\tau$ in $\Gal(L/K)$.
\end{prop}

\begin{proof}[Proof] By virtue of Proposition \ref{commute} and the formula (\ref{formula}), we have \begin{align*}
    \eta(\tau)=\pi_X(\tsigma)^{-1}\circ {}^{\sigma}\eta(\tau)\circ\pi_{{}^{\tau}X}(\tsigma).
\end{align*} We write
\begin{align*}
    \eta(\sigma\tau) &=\bar{\rho}_{\ell,X}(\tsigma\ttau)\circ\pi_X(\tsigma\ttau)^{-1}\\ &=\bar{\rho}_{\ell,X}(\tsigma)\circ\bar{\rho}_{\ell,X}(\ttau)\circ\pi_X(\ttau)^{-1}\circ\pi_{{}^{\tau}X}(\tsigma)^{-1}\\ &=\bar{\rho}_{\ell,X}(\tsigma)\circ\eta(\tau)\circ\pi_{{}^{\tau}X}(\tsigma)^{-1}\\ &=\bar{\rho}_{\ell,X}(\tsigma)\circ\pi_X(\tsigma)^{-1}\circ {}^{\sigma}\eta(\tau)\circ \pi_{{}^{\tau}X}(\tsigma)\circ \pi_{{}^{\tau}X}(\tsigma)^{-1}\\ &=\eta(\sigma)\circ {}^{\sigma}\eta(\tau)
\end{align*}
and this gives the expected result.
\end{proof}

In summary, if $X$ satisfies the condition $(*_\ell)$, we have a $G_L$-equivariant isomorphism $\eta(\sigma):\vlsigmax\to\vlx$ for every $\sigma$ in $\Gal(L/K)$, satisfying the cocycle condition $\eta(\sigma\tau)=\eta(\sigma)\circ {}^{\sigma}\eta(\tau)$. If in particular $X$ satisfies the condition $(**_\ell)$, then $\eta(\sigma)$ is a $G_L$-equivariant isomorphism from $\tlsigmax$ to $\tlx$ with the same cocycle condition. As we would like to study some descent properties for $X$, we want to have isomorphisms (or isogenies) of abelian varieties over $L$ from $\sigmax$ to $X$ with such a cocycle condition. From now on and until the end of the section, assume that $L$ is either a finite field, a function field of positive characteristic different from $2$, or an algebraic number field. We have the following well-known theorem:

\begin{thm}[Isogeny theorem] \label{faltings}
    Let $X$, $Y$ be abelian varieties over $L$. Then the map
    \begin{equation}\label{faltings_eq}\begin{split}
        \Phi :\Hom_L(Y,X)\otimes_{\mathbb{Z}}\mathbb{Z}_{\ell}&\rightarrow \Hom_{G_L}(T_{\ell}Y,\tlx)\\
        f\otimes \alpha&\mapsto \alpha\cdot T_{\ell}f\end{split}
    \end{equation}is an isomorphism of $\mathbb{Z}_\ell$-modules.
\end{thm}

This theorem was proved in the finite field case by Tate (\cite{tate_66}), in the function field case by Zahrin and Mori (\cite{zarhin_74} and \cite{mori_77}) and in the number field case by Faltings (\cite{faltings_86}). It implies that the map
\begin{equation}\label{faltings_eq_tensored}\begin{split}
        \Phi_{\mathbb{Q}_\ell} :\Hom_L^0(Y,X)\otimes_{\mathbb{Q}}\mathbb{Q}_{\ell}&\rightarrow \Hom_{G_L}(V_{\ell}Y,\vlx)\\
        f\otimes \alpha&\mapsto \alpha\cdot V_{\ell}f.\end{split}
\end{equation} is an isomorphism of $\mathbb{Q}_\ell$-vector spaces. A consequence of this theorem is that two abelian varieties $X$ and $Y$ are isogenous if and only if their Tate modules are isomorphic as $G_L$-representations. If $X$ satisfies the condition $(*_\ell)$, then for every $\sigma$ in $\Gal(L/K)$ there exists a $G_L$-equivariant isomorphism $\eta(\sigma):\vlsigmax\to\vlx$. This implies:

\begin{thm}
    If $X$ is an abelian variety over $L$ satisfying the condition $(*_\ell)$ for a prime $\ell$ different from $\mathrm{char}(L)$, then $X$ is isogenous to each of its Galois conjugates $\sigmax$, for $\sigma$ in $\Gal(L/K)$.
\end{thm}

An abelian variety defined over $L$ that is isogenous to each of its $\Gal(L/K)$-conjugates is said to be a \emph{K-abelian variety}. From now on, we ask $X$ to be simple such that $\End_L^0(X)=F$ is a number field.
\begin{rem}
Up until now, we did not use the fact that the extended representation $\bar{\rho}_{\ell,X}$ commutes with the endomorphisms of $X$. However, from now on, this assumption becomes necessary.
\end{rem}

\begin{prop}\label{lemma}
    Assume that $X$ satisfies the condition $(*_\ell)$ for a prime $\ell$ different from $\mathrm{char}(L)$. Then for every $\sigma$ in $\Gal(L/K)$ there exist an isogeny $\eta_0(\sigma)\in\Hom_L^0(\sigmax,X)$ commuting with $F$ and $\alpha_\ell(\sigma)$ in $F_\ell^\times:=(F\otimes\mathbb{Q}_\ell)^\times$, such that $\eta(\sigma)=\alpha_\ell(\sigma)\cdot V_\ell(\eta_0(\sigma))$.
\end{prop}

\begin{proof}
    When we take $Y=\sigmax$, both sides of equation (\ref{faltings_eq_tensored}) are $\mathbb{Q}_\ell$-vector spaces of dimension equal to the rank $r$ of the free $\mathbb{Z}$-module $\Hom_L(\sigmax,X)$, so the morphism $\eta(\sigma)$ will \emph{a priori} be equal to a sum $\sum \alpha_i V_\ell g_i$ with $\alpha_i\in\mathbb{Q}_\ell$, and $g_i\in\Hom_L^0(\sigmax,X)$, for $i=1,\cdots,r$. We required the extended representation $\bar{\rho}_{\ell,X}$ to commute with $F$, and the canonical identification between $\End_L^0(X)$ and $\End_L^0(\sigmax)$ given by $g\mapsto {}^{\sigma}g$ commutes with $\pi_X(\tsigma)$ by formula (\ref{formula}). Thus $\eta(\sigma)$ commutes with $F$ and $\Hom_L^0(\sigmax,X)$ becomes a vector space over $F$ with scalar multiplication defined as 
    \begin{align*}
        F\times\Hom_L^0(\sigmax,X)&\to \Hom_L^0(\sigmax,X)\\
        (f,\mu)&\mapsto f\circ \mu.
    \end{align*}
    Its dimension as a vector space over $F$ is $1$: first of all, its dimension is strictly positive because there exists a non-trivial morphism from $\sigmax$ to $X$. Then, choose a non-zero $\psi\in\Hom_L^0(X,\sigmax)$ and define the $F$-linear form
    \begin{align*}
        \Hom_L^0(\sigmax,X)&\to\End^0(X)=F\\
        \phi&\mapsto\phi\circ \psi.
    \end{align*} The kernel of this map is obviously $\{0\}$, so the map is injective. Because the dimension of $\Hom_L^0(\sigmax,X)$ as an $F$-vector space is greater than $1$, this map is an isomorphism and so its dimension is $1$. As a consequence, there exists $\eta_0(\sigma)\in\Hom_L^0(\sigmax,X)$ and $f_i\in F$ such that $g_i = f_i\circ\eta_0(\sigma)$ for $i=1,\cdots,r$, thus $\Phi_{\mathbb{Q}_\ell}^{-1}(\eta(\sigma))=(\sum f_i\otimes\alpha_i)\cdot \eta_0(\sigma)$, and we set $\alpha_\ell(\sigma):=\sum f_i\otimes\alpha_i\in F_\ell$. Finally, $\alpha_\ell(\sigma)$ lies in $F_\ell^\times$ for every $\sigma\in\Gal(L/K)$ because $\eta(\sigma)$ and $V_\ell(\eta_0(\sigma))$ both have an inverse in the free $F_\ell$-module of rank one $\Hom_{G_L}(\vlx,\vlsigmax)$, and $\eta_0(\sigma)$ commutes with $F$ because $\eta(\sigma)$ does.
\end{proof}

Note that the choice for $\eta_0(\sigma)$ is not unique. However, as $\eta(\mathrm{id})=\mathrm{id}_{V_\ell X}$, we can (and will) set $\eta_0(\mathrm{id})=\mathrm{id}_X$ and $\alpha_\ell(\mathrm{id})=1$ independently of $F$. The morphism $\eta_0(\sigma)$ can always be chosen to be a genuine isogeny from $\sigmax$ to $X$, i.e. a surjective morphism of abelian varieties over $L$ with finite kernel, after multiplying by a suitable rational number.

\begin{rem}\label{Sur_Z}
    If $\End(X)=\mathbb{Z}$ and if $X$ satisfies the condition $(**_\ell)$, then there exists for every $\sigma$ in $\Gal(L/K)$ an isogeny $\eta_0(\sigma)\in\Hom_L(\sigmax,X)$ and $\alpha_\ell(\sigma)\in\mathbb{Z}_\ell$ such that $\eta(\sigma)=\alpha_\ell(\sigma)\cdot T_\ell(\eta_0(\sigma))$. Indeed, we know that $\Hom_L(\sigmax,X)$ is a free $\mathbb{Z}$-module and it is of rank one, because the dimension of $\Hom_L^0(\sigmax,X)$ as a $\mathbb{Q}$-vector space is one.
\end{rem}

\begin{prop}\label{isom}
    Assume that $\End(X)=\mathbb{Z}$ and that $L$ is a number field. If $X$ satisfies the condition $(**_\ell)$ for every prime $\ell$, then $X$ and $\sigmax$ are isomorphic over $L$ for every $\sigma$ in $\Gal(L/K)$.
\end{prop}

\begin{proof}
    Extending the representation for a prime $\ell$ gives an isogeny $\eta_0(\sigma):\sigmax\to X$ for every $\sigma\in\Gal(L/K)$; choose $\eta_0(\sigma)$ to be a generator of $\Hom_L(\sigmax,X)$ as a $\mathbb{Z}$-module. We know that there exists an "inverse isogeny" $\zeta_0(\sigma):X\to\sigmax$ such that $\eta_0(\sigma)\circ\zeta_0(\sigma)=[\mathrm{deg}(\eta_0(\sigma))]_X$ and $\zeta_0(\sigma)\circ\eta_0(\sigma)=[\mathrm{deg}(\eta_0(\sigma))]_{\sigmax}$. There exists $\alpha_\ell(\sigma)\in\mathbb{Z}_\ell$ such that the product $\eta(\sigma)=\alpha_\ell(\sigma)\cdot T_\ell(\eta_0(\sigma))$ is invertible in the $\mathbb{Z}_\ell$-module $\Hom(\tlsigmax,\tlx)$, so $\mathrm{deg}(\eta_0(\sigma))$ cannot be divisible by $\ell$. If the representation extends for every prime $\ell$, then $\mathrm{deg}(\eta_0(\sigma))$ is not divisible by any prime, so $\mathrm{deg}(\eta_0(\sigma))=1$ and $\eta_0(\sigma)$ is an isomorphism for every $\sigma\in\Gal(L/K)$.
\end{proof}

\section{A descent result}\label{partial_descent_section}

Let $K$ be an arbitrary field, let $Y$ be an abelian variety of dimension $g$ defined over the separably closed field $\bar{K}$. Ribet proved in \cite{ribet_1994} that, if $Y$ has \emph{real multiplication}, i.e. the algebra $\End_{\bar{K}}^0(Y)$ is a totally real field $F$ of degree $g$, and if for every $\tsigma$ in $G_K$ there is an $F$-equivariant isogeny from ${}^{\tsigma}Y$ to $Y$, then there exists a $(2,\cdots,2)$-extension $K'$ of $K$, i.e. a Galois extension $K'$ of $K$ with $\Gal(K'/K)\simeq \mathbb{Z}/2\mathbb{Z}\times\cdots\times\mathbb{Z}/2\mathbb{Z}$, such that $Y$ is isogenous to an abelian variety defined over $K'$. A careful study of the article shows that the condition about the degree of $F$ can be dropped, and that the result still holds. We keep the same notations as before, so we will assume that there exists an isogeny $\eta_0(\tsigma)\in\Hom_{\bar{K}}^0({}^{\tsigma}Y,Y)$ for every $\tsigma$ in $G_K$, and will ask that they commute with $F$. There exists a finite Galois field extension of $K$ such that $X$ is defined over this extension; to stay consistent with the rest of the article, we denote by $L$ this extension, so that ${}^{\tsigma}Y$ is identified with $Y$ whenever $\tsigma$ lies in $G_L$. As a consequence, there is only a finite number of isogenies $\eta_0(\tsigma)$ to consider up to multiplication by elements of $F^{\times}$, and after doing a finite field extension, we can and will assume that they are all defined over $L$. We fix once again a separable closure $\bar{L}=\bar{K}$ and we will write $\bar{L}$ instead of $\bar{K}$. The theorem of Ribet, stated in its greater generality, is the following:

\begin{thm}\label{ribet}
    Let $Y$ be a $K$-abelian variety defined over $\bar{L}$, with $\End_{\bar{L}}^0(X)= F$ a totally real number field, and for every $\tsigma$ in $G_K$ assume that there is an isogeny $\eta_0(\tsigma)\in\Hom_{\bar{L}}^0({}^{\tsigma}Y,Y)$ commuting with $F$. Then $Y$ is $F$-equivariantly isogenous to an abelian variety defined over a $(2,...,2)$-extension $K'$ of $K$; in particular the endomorphisms of $Y$ descend to $K'$ too.
\end{thm}

This theorem gives the same result as Theorem (1.2) in \cite{ribet_1994}, so we are going to follow the exact same proof, namely the proofs of Propositions (3.1), (3.2) and (3.3) in his article that directly imply Theorem (1.2), and show that we do not require the field $F$ to be of degree $g$ over $\mathbb{Q}$. Some details will be omitted, see the original article of Ribet for the proof in its integrality. 
\newline For the moment, $F$ can be an arbitrary number field . Denote by $\iota_Y:F\tilde{\to}\End_{\bar{L}}^0(Y)$, resp. by $\iota_{{}^{\tsigma}Y}:F\tilde{\to}\End_{\bar{L}}^0({}^{\tsigma}Y)$ the inclusion (here an isomorphism) of $F$ in the endomorphism algebra of $Y$, resp. ${}^{\tsigma}Y$. Recall from Section \ref{definition_section} that the algebras $\End_{\bar{L}}^0(Y)$ and $\End_{\bar{L}}^0({}^{\tsigma}Y)$ can be canonically identified, so we endow the field $F$ (and in particular $F^\times$) with a trivial action of the group $G_K$.
Explicitly, this means that for $g\in\End_{\bar{L}}^0(Y)$, the elements $\iota_Y^{-1}(g)$ and $\iota_{{}^{\tsigma}Y}^{-1}({}^{\tsigma}g)$ are equal in $F$. For $\tsigma$, $\ttau$ in $G_K$, the composition $\eta_0(\tsigma)\circ {}^{\tsigma}\eta_0(\ttau)\circ\eta_0(\tsigma\ttau)^{-1}$ is a non-zero element of $\End_{\bar{L}}^0(Y)$ and thus of the form $\iota_Y(c(\tsigma,\ttau))$ with $c(\tsigma,\ttau)\in F^\times$. This construction defines a map $(\tsigma,\ttau)\mapsto c(\tsigma,\ttau)$ and a quick computation shows that $c$ is a $2$-cocycle on $G_K$ with values in $F^\times$, seen as a trivial $G_K$-module. We write $c\in Z^2(G_K,F^\times)$ and we denote its image in the cohomology group $H^2(G_K,F^\times)$ by $\gamma$. If we make another choice for the maps $\eta_0(\tsigma)$ and we replace $\eta_0(\tsigma)$ by another $\eta_0'(\tsigma)\in\Hom_{\bar{L}}^0({}^{\tsigma}Y,Y)$ for every $\tsigma\in G_K$, then the corresponding cocycle $c'$ only differs from $c$ by a 2-coboundary with values in $F^\times$. Indeed, $\Hom_{\bar{L}}^0({}^{\tsigma}Y,Y)$ is a vector space of dimension one over $F$ by the proof of Proposition \ref{lemma}, which implies that $\eta_0(\tsigma)$ and $\eta_0'(\tsigma)$ only differ by an element of $F^\times$. So $c$ and $c'$ define the same cohomology class $\gamma$, thus $\gamma$ is well-defined.

    \begin{lem}\label{descente} Let $K'/K$ be a field extension, denote $G_{K'}:=\Gal(\bar{L}/K')$ the absolute Galois group of $K'$. The abelian variety $Y$ is $F$-equivariantly isogenous to an abelian variety $Z_{K'}$ defined over $K'$ if and only if $\gamma$ lies in the kernel of the restriction map $H^2(G_K,F^\times)\to H^2(G_{K'},F^\times)$.\end{lem}
        \begin{proof} This is Proposition $(3.1)$ in Ribet's article, where he asked for $Y$ to have real multiplication. In fact, this lemma is true for an arbitrary field $F$.

        We follow the proof of this proposition as written in Ribet's article. First notice that it is equivalent to prove that $Y$ is $F$-equivariantly isogenous to an abelian variety $Z_K$ defined over $K$ if and only if $\gamma=1$. If $Y$ is isogenous to an abelian variety $Z_K$ defined over $K$, we choose an isogeny $g:Y\to Z=Z_K\times_K \bar{L}$  and set $\mu_0(\tsigma):=g^{-1}\circ {}^{\tsigma}g\in\Hom_{\bar{L}}^0({}^{\tsigma}Y,Y)$ for every $\tsigma\in G_K$. It satisfies $\mu_0(\tsigma)\circ {}^{\tsigma}\mu_0(\ttau)\circ\mu_0(\tsigma\ttau)^{-1}=1$ for every $\tsigma$, $\ttau$ in $G_K$, so the associated cohomology class $\gamma$ is trivial.
        The proof of the converse uses Weil's restriction of scalars of the variety $Y$, whose construction does not depend on whether $Y$ has real multiplication or not. We assumed that $Y$ and all the isogenies $\eta_0(\tsigma)$ were defined over $L$. Following the proof in Ribet's article, and writing $R_Y:=\mathrm{Res}_{L/K}Y$, we find that $\End_K ^0 (R_Y)=F[\Gal(L/K)]=\oplus_{\sigma} F\cdot\eta_0(\tsigma)$ if $\gamma=1$, where the direct sum runs over the elements of $\Gal(L/K)$ and $\tsigma$ is an arbitrary extension of $\sigma$ to $G_K$, without any assumption needed on the field $F$. The field $F$ is a direct factor of $\End^0_K(R_Y)$ as algebras and we denote by $Z_K$ the subvariety of $R_Y$ corresponding to this direct factor; by this we mean that $Z_K$ is the subvariety of $R_Y$ whose endomorphism algebra $\End^0_K(Z_K)$ is the direct factor $F$ of $\End^0_K(R_Y)$. Furthermore, we know that $R_{Y,L}=R_Y\times_K L$ is isomorphic to the product $\prod_{\sigma\in\Gal(L/K)}{}^{\sigma}Y$ because the extension $L|K$ is separable, and as a consequence $Z=Z_K\times_K L$ is isogenous to a product of some of the Galois conjugates ${}^{\sigma}Y$. We have
        \begin{align*}
            \End^0_K(Z_K)=F=\Hom^0_K(Z_K,R_Y)=\Hom^0_L(Z,Y)
        \end{align*} by property of the restriction of scalars. But $F=\End_{\bar{L}}^0(Y)$ too so $Z$ is isogenous to exactly one Galois conjugate ${}^{\sigma}Y$. As a conclusion, the abelian variety $Z_K$ is defined over $K$ and $F$-equivariantly isogenous to $Y$. The varieties $Z_K$ and $Y$ are isogenous over $L$ because the isogenies $\eta_0(\tsigma)$ are defined over $L$ by assumption.\end{proof}
    \begin{rem}\label{gal_L_K}
        Under our assumptions, the maps $\eta_0(\tsigma)$ are defined over $L$ and only depend on the restriction of $\tsigma$ to $L$, and $\gamma$ can be seen as an element of $H^2(\Gal(L/K),F^{\times})$. Following the same arguments, the abelian variety $Y$ is $F$-equivariantly isogenous over $L$ to an abelian variety $Z_{K}$ defined over $K$ if and only if $\gamma$ is trivial in the group $H^2(\Gal(L/K),F^\times)$.
    \end{rem}

    From now on we require $F$ to be a totally real field, i.e. we only consider simple abelian varieties of Type I in Albert's classification. This excludes the case where $L$ is a finite field, because a simple abelian variety defined over a finite field is either of Type III or Type IV in Albert's classification (see \cite{oort_2008}, (6.1)). Let $Y'$, $Y''$ be other abelian varieties defined over $\bar{L}$, whose endomorphism algebra is $F$ for both and that are $F$-equivariantly isogenous to each of their $G_K$-conjugates. Ribet defines, for $\theta_Y:Y\to Y^\vee$ (resp. $\theta_{Y'}:Y'\to (Y')^\vee$) a polarization of $Y$ (resp. $Y'$) as an abelian variety, and for $\mu\in\Hom_{\bar{L}}^0(Y',Y)$ any $F$-equivariant morphism, an endomorphism of $Y$ up to isogeny $\mathrm{deg}(\mu)$ defined as
    \begin{align}
        \mathrm{deg}(\mu)=\mu\circ\theta_{Y'}^{-1}\circ\mu^\vee\circ\theta_Y\in\End^0_{\bar{L}}(Y).
    \end{align}
    We can identify $\mathrm{deg}(\mu)$ with an element of $F^\times$ via the isomorphism $\iota_Y$. Note that $\deg(\mu)$ depends on the choice of the polarizations. If $\mu\in\End_{\bar{L}}^0(Y)$ is given by an element $c\in F$, then $\mathrm{deg}(\mu)=c^2$, because the Rosati involution given by the polarization $\theta_Y$ is the identity, as $F$ is a totally real field. If $\mu\in\Hom_{\bar{L}}^0(Y',Y)$ and $\lambda\in\Hom_{\bar{L}}^0(Y'',Y')$ are $F$-equivariant isogenies of abelian varieties, then $\mathrm{deg}(\mu\circ\lambda)=\mathrm{deg}(\mu)\cdot\mathrm{deg}(\lambda)$. This fact is easily proven with this diagram (not commutative):
    \begin{center} $\begin{CD}
    Y'' @> \lambda >> Y' @> \mu >> Y\\
    @V \theta_{Y''} VV @V \theta_{Y'} VV @VV \theta_Y V \\
    (Y'')^\vee @<< \lambda^\vee < (Y')^\vee @<< \mu^\vee < Y^\vee
    \end{CD}$
    \end{center}
    by checking that
    \begin{align*}
        \iota_Y(\mathrm{deg}(\mu\circ\lambda))&=\mu\circ\lambda\circ\theta_{Y''}^{-1}\circ(\mu\circ\lambda)^\vee\circ\theta_Y\\
        &=\mu\circ\lambda\circ\theta_{Y''}^{-1}\circ\lambda^\vee\circ\mu^\vee\circ\theta_Y\\
        &=\mu\circ\lambda\circ\theta_{Y''}^{-1}\circ\lambda^\vee\circ\theta_{Y'}\circ\theta_{Y'}^{-1}\circ\mu^\vee\circ\theta_Y\\
        &=\mu\circ\iota_{Y'}(\mathrm{deg}(\lambda))\circ\theta_{Y'}^{-1}\circ\mu^\vee\circ\theta_Y\\
        &=\iota_Y(\mathrm{deg}(\lambda))\circ\mu\circ\theta_{Y'}^{-1}\circ\mu^\vee\circ\theta_Y\\
        &=\iota_Y(\mathrm{deg}(\lambda))\circ\iota_Y(\mathrm{deg}(\mu))\\
        &=\iota_Y(\mathrm{deg}(\mu)\cdot\mathrm{deg}(\lambda)).
    \end{align*}

    \begin{lem}\label{square} If $\End_{\bar{L}}^0(Y)=F$ is a totally real number field, then the order of the cohomology class $\gamma$ associated to $Y$ is at most two.\end{lem}
        \begin{proof} This corresponds to Proposition (3.2) in Ribet's article.  We write $d_{\tsigma}=\mathrm{deg}(\eta_0(\tsigma))\in F^\times$ and we recall that $\gamma$ is the cohomology class of the $2$-cocycle $c$ on $G_K$ with values in $F^\times$, with $c(\tsigma,\ttau)=\eta_0(\tsigma)\circ {}^{\tsigma}\eta_0(\ttau)\circ\eta_0(\tsigma\ttau)^{-1}$. By taking degrees on both sides of this equality and by using the two properties we stated on the degree, we get $c(\tsigma,\ttau)^2=\frac{d_{\tsigma} {}^{\tsigma}d_{\ttau}}{d_{\tsigma\ttau}}=\frac{d_{\tsigma} d_{\ttau}}{d_{\tsigma\ttau}}$ (recall that $F^\times$ is endowed with a trivial action of $G_K$), thus $c^2$ is a coboundary on $G_K$ with values in $F^\times$. This means that $\gamma^2=1$ in the cohomology group $H^2(G_K,F^\times)$. \end{proof}
    \begin{lem}\label{extension}
        Let $F$ be a totally real field. If $\gamma$ is an element of order $2$ in $H^2(G_K,F^\times)$, then there exists a $(2,\cdots,2)$-extension $K'$ of $K$ such that the image of $\gamma$ in $H^2(G_{K'},F^\times)$ is trivial.
    \end{lem}
    \begin{proof}
        This statement is Theorem (3.3) in Ribet's article. The proof is exactly the same as it only relies on abstract cohomological algebra, and does not depend anymore on the abelian variety $Y$ nor on the degree of $F$. We give a short sketch of the proof here. Ribet defines $P=F^{\times}/\{\pm 1\}$ and shows that this defines a split exact sequence
        \begin{align}
            0\to\{\pm 1\}\to F^{\times}\to P\to 0.
        \end{align}As a consequence, there is an isomorphism of abelian groups $F^{\times}\simeq\{\pm 1\}\times P$ which induces a decomposition
        \begin{align}
            H^2(G_K,F^{\times})[2]\simeq H^2(G_K,\{\pm 1\})\times H^2(G_K,P)[2]
        \end{align}where the notation $[2]$ refers to the kernel of the multiplication by $2$ in the cohomology group. Then the cohomology class $\gamma$ can be written as the product of its projections $\gamma_\pm$ in $H^2(G_K,\{\pm 1\})$ and $\bar{\gamma}$ in $H^2(G_K,P)[2]$, and is trivial if and only if both $\gamma_\pm$ and $\bar{\gamma}$ are trivial. It remains to show that $\gamma_\pm$ and $\bar{\gamma}$ become trivial when replacing $G_K$ by $G_{K'}$, where $K'$ is a $(2,\cdots,2)$-extension of $K$.
        \newline To prove this fact for $\bar{\gamma}$, Ribet considers the exact sequence
        \begin{align}
            1\to P\to P\to P/P^2\to 1
        \end{align}where the second arrow is given by $x\mapsto x^2$, and the associated long exact cohomology sequence gives an isomorphism
        \begin{align}
            \Hom(G_K,P/P^2)\simeq H^2(G_K,P)[2].
        \end{align}If $\bar{\gamma}$ corresponds to a homomorphism $\varphi:G_K\to P/P^2$, then its kernel is a normal subgroup of $G_K$ that corresponds to a $(2,\cdots,2)$-extension of $K$, and $\bar{\gamma}$ becomes trivial over this extension.
        \newline Finally, there remains to prove this fact for $\gamma_\pm$. The group $H^2(G_K,\{\pm 1\})$ can be identified with the elements of $2$-torsion in the Brauer group of $K$ (see \cite{serre_1997}, p.84, Remarque 1, with $n=2$), and it is equal to ${0}$ if $K$ has characteristic $2$ (\cite{serre_1997}, p.86). Assume that the characteristic of $K$ is different from $2$. It is known that $\mathrm{Br}(K)[2]$ is generated by the classes of quaternion algebras over $K$ (\cite{merkurev_1981}). Every quaternion algebra over $K$ is split by some quadratic extension of $K$, so the class $\gamma_\pm$ becomes trivial after extending $K$ to a $(2,\cdots,2)$-extension of $K$. This concludes the proof of the lemma.
    \end{proof}
    \begin{proof}[Proof of Theorem \ref{ribet}] Combining these three lemmas concludes the proof of the theorem. \end{proof}
Let $L$ be either a number field or a function field of characteristic different from $2$. Let $X$ be an abelian variety defined over $L$, satisfying the condition $(*_\ell)$ for a prime $\ell$, and assume that $\End_{L}^0(X)=\End_{\bar{L}}^0(X)=F$ is a totally real field. Consider the variety $X_{\bar{L}}=X\times_L \bar{L}$ obtained by extension of scalars. Because of Proposition \ref{lemma}, there exists for every $\tsigma$ in $G_K$ an isogeny $\eta_0(\tsigma)\in\Hom_{\bar{L}}^0({}^{\tsigma}X_{\bar{L}},X_{\bar{L}})$ (that coincides with $\eta_0(\tsigma_{\mid L})\in\Hom_L({}^{\tsigma_{\mid L}}X,X)$ for every $\tsigma$ in $G_K$) commuting with $F$. Then $X_{\bar{L}}$ satisfies the conditions of Theorem \ref{ribet} and we have as a consequence:
\begin{cor}
    Assume that $X$ satisfies the condition $(*_\ell)$ for a prime $\ell\neq\mathrm{char}(L)$, and that $\End_L^0(X)=\End_{\bar{L}}^0(X)=F$ is a totally real field. If the order of the group $\Gal(L/K)$ is odd, then the abelian variety $X$ is isogenous over $L$ to an abelian variety $X_K$ defined over $K$.
\end{cor}
\begin{proof}
    Lemma \ref{square} shows that the order of the cohomology class $\gamma$ associated to $X_{\bar{L}}$ is at most $2$. By the proof of Lemma \ref{extension}, we can write $\gamma$ as the product of its two projections $\gamma_\pm\in H^2(G_K,\{\pm 1\})$ and $\bar{\gamma}\in H^2(G_K,P)[2]$, and $\gamma$ is trivial if and only if its two projections are trivial. However, by assumption, $X_{\bar{L}}$ comes from an abelian variety defined over $L$ and the restriction of $\gamma$ to the cohomology group $H^2(G_L,F^{\times})$ is trivial. This implies that the restrictions of $\gamma_\pm$ and $\bar{\gamma}$ to $H^2(G_L,\{\pm 1\})$ and to $H^2(G_L,P)$ are trivial too.
    Assume that the class $\gamma_\pm$ is not trivial already in $H^2(G_K,\{\pm 1\})$, and denote by $\Gamma$ the corresponding product of classes of quaternion algebras in the Brauer group $\mathrm{Br}(K)$. The order of $\Gamma$ in $\mathrm{Br}(K)$ is equal to $2$. By \cite{gille_2006}, Proposition 4.5.13, the order of $\Gamma$ divides the index of the simple algebra of which it is the class, thus its index is even. By assumption, $L$ is an extension of $K$ of odd degree that trivializes $\Gamma$ in the Brauer group. However, \cite{gille_2006}, Proposition 4.5.8. states that the index of the central simple $K$-algebra of which $\Gamma$ is the class in $\mathrm{Br}(K)$ is equal to the greatest common divisor of the degrees of all the separable finite field extensions of $K$ that trivialize $\Gamma$. This implies that any separable finite field extension that trivializes $\Gamma$ must have an even degree. As a conclusion, the class $\gamma_{\pm}$ is already trivial over $K$.
    Finally, for the class $\bar{\gamma}\in H^2(G_K,P)[2]$, the proof of Lemma \ref{extension} shows that it can be identified with a homomorphism $\varphi:G_K\to P/P^2$, such that the kernel of $\varphi$ corresponds to the absolute Galois group of a $(2,\cdots,2)$-extension $K'$ of $K$. As the restriction of $\bar{\gamma}$ to $H^2(G_L,P)$ is trivial, we deduce that the Galois group $G_L$ is contained in the kernel of $\varphi$, i.e $G_L$ is contained in $G_{K'}$ and $L$ is a field extension of $K'$. However, by assumption $L$ is an extension of $K$ of odd degree, so the only suitable $(2,\cdots,2)$-extension $K'$ of $K$ is $K$ itself. So $X_{\bar{L}}$ is isogenous over $L$ to an abelian variety defined over $K$.
\end{proof}

\begin{rem}
    The preceding corollary shows that, if the order of the group $\Gal(L/K)$ is odd and if $\End_L^0(X)=\End_{\bar{L}}^0(X)=F$ is a totally real field, then it is equivalent to extend the representation for one prime $\ell$ and for every prime $\ell$ different from the characteristic of $L$. Indeed, recall that, if $X$ is isogenous to an abelian variety $Y_K$ defined over $K$, then all of its associated $\ell$-adic representations can be extended to the group $G_K$ by using the isogenies from $\sigmax$ to $X$.
\end{rem}

\section{Extension of the representation for all primes}\label{all_l}
Let $L$ be either a number field or a function field of positive characteristic different from $2$. Assume that $\End_L^0(X)=F$ is a totally real number field. In this section, we do not require $\End_{\bar{L}}^0(X)=F$ anymore. If the abelian variety $X$ satisfies the condition $(*_\ell)$ for a prime $\ell\neq\mathrm{char}(L)$ and the extended representation commutes with $F$, section $\ref{extension_section}$ shows that there exist for every $\sigma$ in $\Gal(L/K)$ an isomorphism of Tate modules $\eta_\ell(\sigma):\vlsigmax\to\vlx$, an isogeny $\eta_0(\sigma)\in\Hom_L^0(\sigmax,X)$ commuting with $F$, and $\alpha_\ell(\sigma)\in F_\ell^\times$ such that $\eta_\ell(\sigma)=\alpha_\ell(\sigma)\cdot V_\ell(\eta_0(\sigma))$. From the equality
\begin{align*}
    \eta_\ell(\sigma\tau)=\eta_\ell(\sigma)\circ {}^{\sigma}\eta_\ell(\tau),
\end{align*}
proved in Proposition $\ref{cocycle}$, we obtain the following equality on the corresponding isogenies (the isogenies $\eta_0(\sigma)$ are $F_\ell$-linear):
\begin{align}\begin{split}
    \alpha_\ell(\sigma\tau)\cdot V_\ell(\eta_0(\sigma\tau))&=\alpha_\ell(\sigma)\cdot V_\ell(\eta_0(\sigma))\circ \alpha_\ell(\tau)\cdot V_\ell({}^{\sigma}\eta_0(\tau))\\
    &=\alpha_\ell(\sigma) \alpha_\ell(\tau)\cdot V_\ell(\eta_0(\sigma)\circ {}^{\sigma}\eta_0(\tau)).
\end{split}\end{align}
This implies the following equality in $\Hom^0_L({}^{\sigma\tau}X,X)\otimes\mathbb{Q}_\ell$, for every $\sigma$, $\tau$ in $\Gal(L/K)$:
\begin{align}
    \alpha_\ell(\sigma\tau)\cdot \eta_0(\sigma\tau)=\alpha_\ell(\sigma) \alpha_\ell(\tau)\cdot \eta_0(\sigma)\circ {}^{\sigma}\eta_0(\tau),\end{align}and with the same notation as in section \ref{partial_descent_section}, this gives:
    \begin{align}\begin{split}\frac{\alpha_\ell(\sigma\tau)}{\alpha_\ell(\sigma) \alpha_\ell(\tau)}&=\eta_0(\sigma)\circ {}^{\sigma}\eta_0(\tau)\circ \eta_0(\sigma\tau)^{-1}\\
    \frac{\alpha_\ell(\sigma\tau)}{\alpha_\ell(\sigma) \alpha_\ell(\tau)}&=c(\sigma,\tau).\end{split}
\end{align}
Thus $c$ is a $2$-coboundary of $\Gal(L/K)$ with values in $F_\ell^\times$, so its class in the cohomology group $H^2(\Gal(L/K),F_\ell^\times)$ is trivial. We can say more: $F_\ell$ is the product of the field completions $F_{\lambda_i}$ where $\lambda_i$, $i=1,...,r$, are the primes of $F$ lying above $\ell$. We can write
\begin{align}
    c(\sigma,\tau)=\frac{(\alpha_{\lambda_1}(\sigma\tau),\cdots,\alpha_{\lambda_r}(\sigma\tau))}{(\alpha_{\lambda_1}(\sigma),\cdots,\alpha_{\lambda_r}(\sigma))\cdot(\alpha_{\lambda_1}(\tau),\cdots,\alpha_{\lambda_r}(\tau))}
\end{align} with $\alpha_{\lambda_i}:\Gal(L/K)\to F_{\lambda_i}^\times$. Now $c(\sigma,\tau)$ is an element of $F^\times$ so all the fractions $\frac{\alpha_{\lambda_i}(\sigma\tau)}{\alpha_{\lambda_i}(\sigma)\alpha_{\lambda_i}(\tau)}$ are in $F^\times$ and are equal, for $i=1,\cdots,r$. Finally we have
\begin{align}
    c(\sigma,\tau)=\frac{\alpha_{\lambda_i}(\sigma\tau)}{\alpha_{\lambda_i}(\sigma)\alpha_{\lambda_i}(\tau)}\in B^2(\Gal(L/K),F_{\lambda_i}^\times)
\end{align}
for every prime $\lambda_i$ of $F$ dividing $\ell$, where $B^2$ denotes the $2$-coboundaries.
\newline We would like to find a condition to determine if the abelian variety $X$ descends to $K$ up to isogeny. This is the case if $c\in B^2(\Gal(L/K),F^{\times})$ by Remark \ref{gal_L_K}, and this condition is met for example if there exists a prime $\lambda_i$ such that the $\alpha_{\lambda_i}(\sigma)$'s lie in $F^\times$ for every $\sigma$ in $\Gal(L/K)$. \textit{A priori} this is not the case; however, assume that the abelian variety $X$ satisfies the condition $(*_{\ell'})$ for another prime $\ell'$, different from $\ell$ and from the characteristic of $L$. If we consider the induced isomorphisms of Tate modules $\eta_{\ell'}(\sigma):V_{\ell'}{}^{\sigma}X\to V_{\ell'}X$, then for every $\sigma$ in $\Gal(L/K)$ there exists $\alpha_{\ell'}(\sigma)\in F_{\ell'}^\times$ such that $\eta_{\ell'}(\sigma)=\alpha_{\ell'}(\sigma)\cdot V_{\ell'}(\eta_0(\sigma))$ because of Proposition \ref{lemma}. Moreover, the isomorphisms of Tate modules $\eta_{\ell'}(\sigma)$ satisfy the cocycle condition $\eta_{\ell'}(\sigma\tau)=\eta_{\ell'}(\sigma)\circ {}^{\sigma}\eta_{\ell'}(\tau)$ too. As a consequence, the equality
\begin{equation}
\alpha_\ell(\sigma\tau)=c(\sigma,\tau)\alpha_\ell(\sigma)\alpha_\ell(\tau)
\end{equation}
holds for every prime $\ell$ such that $X$ satisfies the condition $(*_\ell)$, and $c(\sigma,\tau)\in F^\times$ is independent of $\ell$ because the isogenies $\eta_0(\sigma)$ are independent of $\ell$. If $X$ satisfies the condition $(*_\ell)$ for every $\ell$ in a given set of primes not containing the characteristic of $L$, we deduce that, for every prime $\lambda$ of $F$ dividing an element in this set, there exists a map $\alpha_\lambda:\Gal(L/K)\to F_\lambda^\times$ such that
\begin{align}\label{eq_alpha}
    \alpha_\lambda(\sigma\tau)=c(\sigma,\tau)\alpha_\lambda(\sigma)\alpha_\lambda(\tau).
\end{align}
Fix an algebraic closure $\bar{F}$ of $F$. Let $n$, $s$ be natural integers, denote by $\zeta_{2^s}$ a primitive $2^s$-th root of unity in $\bar{F}$ and write $\eta_s=\zeta_{2^s}+\zeta_{2^s}^{-1}$. We say that $F$ satisfies the condition $(GW_n)$ if there exists an integer $s\geq 1$ such that $2^{s+1}$ divides $n$ and $F$ contains $\eta_s$ but $-1$, $2+\eta_{s}$ and $-(2+\eta_{s})$ are not squares in $F$, and if $-1$, $2+\eta_{s}$ and $-(2+\eta_{s})$ are squares in the completion $F_\lambda$ for every (necessarily 2-adic) prime $\lambda$ of $F$. The theorem of Grunwald-Wang states that, if $F$ does not satisfy the (very restrictive) condition $(GW_n)$, then an element of $F$ that is a $n$-th power in every completion $F_\lambda$ of $F$, except for a finite set of primes $S$ that does not contain any prime dividing $2$, is a $n$-th power in $F$ (see \cite{artin_tate} Chapter X for example).

\begin{thm}\label{descente_complete}
    Let $n$ be the order of the group $\Gal(L/K)$, and assume $\End_L^0(X)=F$ is a totally real field that does not satisfy the condition $(GW_n)$. If $X$ satisfies the condition $(*_\ell)$ for every prime $\ell$ different from $\mathrm{char}(L)$, except a finite set of primes $S$ that does not contain $2$, then there exists an abelian variety $X_K$ defined over $K$ such that $X$ is isogenous over $L$ to $X_K$, and $\End^0_K(X_K)=F$.
\end{thm}
Denote by $\Sigma$ the set of all primes $\lambda$ of $F$ that do not divide $\mathrm{char}(L)$ nor the primes contained in $S$. Before proving the theorem, we need the following lemma:
    \begin{lem}
    There exists a prime $\lambda'$ in $\Sigma$ such that the groups of $n$-th roots of unity $\mu_n(F)$ and $\mu_n(F_{\lambda'})$ are equal.
    \end{lem}
    \begin{proof}
        If $n=2$, this equality obviously holds for every prime $\lambda$. Otherwise, let $\zeta_n$ be a primitive $n$-th root of unity in $\bar{F}$. The extension $F(\zeta_n)\mid F$ is Galois, its Galois group is abelian and denoted by $H\subset (\mathbb{Z}/n\mathbb{Z})^\times$. If $\lambda$ is a place of $F$ and $\Lambda$ an unramified place of $F(\zeta_n)$ above $\lambda$, denote by $\mathrm{Frob}_\lambda$ the associated Frobenius element in the decomposition group $\Gal(F(\zeta_n)_\Lambda/F_\lambda)$. We required $F$ to be a totally real field, so the only roots of unity in $F$ are $1$ and $-1$. The complex conjugation in $H$ acts non-trivially on every root of unity in $F(\zeta_n)$ different from $1$ and $-1$. Chebotarev's density Theorem implies that there exist infinitely many primes of $F$ whose associated Frobenius is conjugated to the complex conjugation in $H$, and thus acts non-trivially on every $n$-th root of unity that is not in $F$. This proves that there exists a prime $\lambda'$ in $\Sigma$ such that $\mu_n(F)=\mu_n(F_{\lambda'})$
    \end{proof}

\begin{proof}[Proof of Theorem \ref{descente_complete}]
    Let $\sigma\in\Gal(L/K)$. From (\ref{eq_alpha}), we have the following equality for any positive integer $m$, for any prime $\lambda$ in $\Sigma$:
    \begin{align*}
        \alpha_\lambda(\sigma^m)=\alpha_\lambda(\sigma\cdot\sigma^{m-1})=c(\sigma,\sigma^{m-1})\alpha_\lambda(\sigma)\alpha_\lambda(\sigma^{m-1}).
    \end{align*} Thus, by induction:
    \begin{align*}
        \alpha_\lambda(\sigma^m)=c(\sigma,\sigma^{m-1})c(\sigma,\sigma^{m-2})\cdots c(\sigma,\sigma)\alpha_\lambda(\sigma)^m.
    \end{align*} The product $c(\sigma,\sigma^{m-1})c(\sigma,\sigma^{m-2})\cdots c(\sigma,\sigma)$ is a non-zero element of $F=\End_L^0(X)$ and is independent of the prime $\lambda$. For $m$ equal to $n$, the order of the group $\Gal(L/K)$, we have the equality $\alpha_\lambda(\sigma^n)=\alpha_\lambda(\mathrm{id})=1=f\cdot\alpha_\lambda(\sigma)^n$ with $f=c(\sigma,\sigma^{n-1})c(\sigma,\sigma^{n-2})\cdots c(\sigma,\sigma)\in F$, that is $\alpha_\lambda(\sigma)^n=1/f\in F^\times$ for all primes $\lambda$ in $\Sigma$. The number $1/f$ is a $n$-th power in the completion $F_\lambda$ for every $\lambda$ in $\Sigma$ so it is a $n$-th power in $F$ too by the Theorem of Grunwald-Wang.

    Denote by $g$ an element of $F^\times$ such that $g^n=1/f$. We want to show that there exists a prime $\lambda$ such that $\alpha_\lambda(\sigma)$ lies in $F^\times$. The set of elements $s$ of $F$ such that $s^n=1/f$ is equal to $\{\zeta\cdot g,\zeta\in\mu_n(F)\}$ and the set of elements $s$ of $F_\lambda$ such that $s^n=1/f$ is equal to $\{\zeta\cdot g,\zeta\in\mu_n(F_\lambda)\}$. The previous lemma shows that there exists a prime $\lambda'$ in $\Sigma$ such that $\mu_n(F)=\mu_n(F_{\lambda'})$, and thus that $\alpha_{\lambda'}(\sigma)$ lies in $F^\times$. The choice of the prime $\lambda'$ does not depend on $\sigma\in\Gal(L/K)$, so $\alpha_{\lambda'}(\sigma)$ lies in $F^\times$ for every $\sigma\in\Gal(L/K)$. As a consequence, $c(\sigma,\tau)=\frac{\alpha_{\lambda'}(\sigma\tau)}{\alpha_{\lambda'}(\sigma)\alpha_{\lambda'}(\tau)}$ is a $2$-coboundary with coefficients in $F^\times$, so $\gamma$ is the trivial class in $H^2(\Gal(L/K),F^\times)$. Because of Remark \ref{gal_L_K}, $X$ is $F$-equivariantly isogenous over $L$ to an abelian variety $X_K$ defined over $K$, with $\End^0_K(X_K)=F$.
\end{proof}

\begin{eg}
The field $F=\mathbb{Q}$ does not satisfy the condition $(GW_n)$ for any integer $n$, so the preceding theorem is true for any finite Galois extension $L/K$ when the endomorphisms of $X$ are trivial. The preceding theorem also holds for every Galois extension $L$ of $K$ such that $8$ does not divide the order of the group $\Gal(L/K)$, independently of the totally real field $F$ (see \cite{artin_tate}, Chapter X).
\end{eg}

We cannot expect the descent to be an isomorphism in this case, because there exists \emph{a priori} no genuine isomorphism from $\sigmax$ to $X$. However, we can get a better result if $X$ satisfies the condition $(**_\ell)$ for every prime $\ell$. Assume that $X$ is defined over a number field $L$ and that $X$ satisfies $(**_\ell)$ for every prime $\ell$: by Proposition \ref{isom}, if $\End(X)=\mathbb{Z}$, we have for every $\sigma$ in $\Gal(L/K)$  an isomorphism $\eta_0(\sigma):\sigmax\tilde{\to} X$ over $L$, and for every prime $\ell$ an $\ell$-adic integer $\alpha_\ell(\sigma)\in\mathbb{Z}_\ell$ such that $\eta_\ell(\sigma)=\alpha_\ell(\sigma)\cdot T_\ell(\eta_0(\sigma))$. We have now for every prime $\ell$ the equality
\begin{align*}
    \alpha_\ell(\sigma\tau)=c(\sigma,\tau)\alpha_\ell(\sigma)\alpha_\ell(\tau),
\end{align*}
with $c(\sigma,\tau)\in\{\pm 1\}$ an automorphism of $X$; the $2$-cocycle $c\in Z^2(\Gal(L/K),\{\pm 1\})$ defines a class $\gamma$ in the cohomology group $H^2(\Gal(L/K),\{\pm 1\})$.
\begin{thm}
    If $\End(X)=\mathbb{Z}$, if $L$ is a number field and if $X$ satisfies the condition $(**_\ell)$ for every prime number $\ell$, then the abelian variety $X$ is isomorphic over $L$ to an abelian variety $X_K$ defined over $K$, with $\End^0_K(X_K)=F$.
\end{thm}
\begin{proof}
    The proof follows the strategy of Theorem \ref{descente_complete}. Let $\sigma\in\Gal(L/K)$. For every positive integer $m$ we have the following equality:
    \begin{align*}
        \alpha_\ell(\sigma^m)&=c(\sigma,\sigma^{m-1})c(\sigma,\sigma^{m-2})\cdots c(\sigma,\sigma)\alpha_\ell(\sigma)^m\end{align*} and thus, for $m=n$ the order of the group $\Gal(L/K)$,
        \begin{align*}\alpha_\ell(\mathrm{id})=1=q\cdot \alpha_\ell(\sigma)^n
    \end{align*}
    with $q=\pm 1$. We see that $\alpha_\ell(\sigma)$ is a $2n$-th root of unity in $\mathbb{Z}_\ell$. For $\ell=2$, the group of roots of unity in $\mathbb{Z}_\ell$ is the same as the group of roots of unity in $\mathbb{Z}$. This implies that $\alpha_2(\sigma)\in\{\pm 1\}$ for every $\sigma$ in $\Gal(L/K)$ and thus that $\gamma$ is trivial in $H^2(\Gal(L/K),\{\pm 1\})$.
    \newline Now, after multiplying every $\eta_0(\sigma)$ with a suitable element of $\{\pm 1\}=\Aut(X)$, the isomorphisms $\eta_0(\sigma)$ satisfy the cocycle condition $\eta_0(\sigma\tau)=\eta_0(\sigma)\circ {}^{\sigma}\eta_0(\tau)$ for every $\sigma$, $\tau\in\Gal(L/K)$. Thus $X$ is isomorphic over $L$ to an abelian variety defined over $K$. 
\end{proof}

\section*{Acknowledgements}
I would like to thank Rutger Noot for guiding me during the writing of this article, and for correcting my mistakes. I also thank Thomas Agugliaro for our mathematical discussions.

\printbibliography[title={References}]

\end{document}